\documentclass{article} 

\usepackage[utf8x]{inputenc}

\usepackage{amsmath,amssymb,amstext,amsthm}
\usepackage{amsfonts}
\usepackage{graphicx,epstopdf,epsfig,multirow,epic,bm}
\usepackage{color}
\usepackage{multicol}
\usepackage{algorithm}
\usepackage{algorithmic}

\theoremstyle{definition}
\newtheorem{thm}{Theorem}

\newtheorem{lem}[thm]{Lemma}

\theoremstyle{definition}
\newtheorem{defn}{Definition}
\theoremstyle{definition}

\theoremstyle{definition}

\theoremstyle{definition}

\theoremstyle{definition}

\theoremstyle{definition}

\DeclareGraphicsExtensions{.eps,.eps.gz}
\normalsize \rm
\newcommand{\qqed}{\hfill $\square$}

\begin{document}

\thispagestyle{empty}

\begin{center}
{\huge  \textbf{Ice-Flower Systems And Star-graphic Lattices}}

\vskip 0.4cm

{\large Bing \textsc{Yao}$^{1}$, Hongyu \textsc{Wang}$^{2,\dag}$, Xia \textsc{Liu}$^{3,\ddagger}$, Xiaomin \textsc{Wang}$^2$, Fei \textsc{Ma}$^{2}$, Jing \textsc{Su}$^2$, Hui \textsc{Sun}$^{2}$}\\

\vskip 0.4cm
{1. College of Mathematics and Statistics, Northwest Normal University, Lanzhou, 730070 CHINA\\
2. School of Electronics Engineering and Computer Science, Peking University, Beijing, 100871, CHINA\\
3. School of Mathematics and Statistics, Beijing Institute of Technology, Beijing, 100081, CHINA\\
Corresponding authors: $^{\dag}$ Hongyu Wang: why1988jy@163.com; $^{\ddagger}$ Xia Liu: 1076641204@qq.com}
\end{center}

\vskip 0.6cm

\begin{abstract}
Lattice theory  has been believed to resist classical computers and quantum computers. Since there are connections between traditional lattices and graphic lattices, it is meaningful to research graphic lattices. We define the so-called ice-flower systems by our uncolored or colored leaf-splitting and leaf-coinciding operations. These ice-flower systems enable us to construct several star-graphic lattices. We use our star-graphic lattices to express some well-known results of graph theory and compute the number of elements of a particular star-graphic lattice. For more researching ice-flower systems and star-graphic lattices we propose Decomposition Number String Problem, finding strongly colored uniform ice-flower systems and connecting our star-graphic lattices with traditional lattices.\\

\end{abstract}

\vskip 1cm

\section{Introduction and preliminary }

\subsection{Investigation background}

There are many important classes of cryptographic systems beyond RSA and DSA and ECDSA, and they are believed to resist classical computers and quantum computers, such as Hash-based cryptography, Code-based cryptography, Lattice-based cryptography, Multivariate-quadratic-equations cryptography, Secret-key cryptography pointed in \cite{Bernstein-Buchmann-dahmen-quantum-2009}. Lattice theory has a wide range of applications in cryptanalysis and has been believed to resist classical computers and quantum computers.

For a given \emph{number string} $D$ shown in (\ref{eqa:0-example}), we cut it into 36 segments to produce a \emph{Topcode-matrix} $T_{code}(T)$ shown in Fig.\ref{fig:0-example} (a), and then we find a colored tree $T$ (called a \emph{Topsnut-gpw}) shown in Fig.\ref{fig:0-example} (b) by means of the Topcode-matrix $T_{code}(T)$, and we use the colored tree $T$ to make a \emph{topological vector} $V_{ec}(T)$ shown in Fig.\ref{fig:0-example} (c).
\begin{equation}\label{eqa:0-example}
{
\begin{split}
D=21262432252222746922221188132020151012172o201914162120182316
\end{split}}
\end{equation}

\textbf{Decomposition Number String Problem.} Cut a number string $D=c_1c_2\cdots c_n$ (like that in (\ref{eqa:0-example})) with $c_j\in [0,9]$ into $3\times q$ segments to form a Topcode-matrix $T_{code}$, and find a Topsnut-gpw $G$ having its own Topcode-matrix $T_{code}(G)=T_{code}$.

\vskip 0.2cm

\begin{figure}[h]
\centering
\includegraphics[width=16cm]{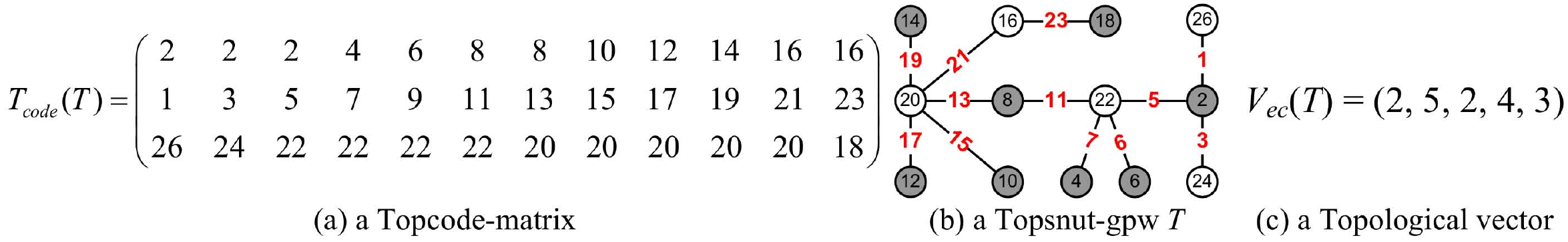}\\
\caption{\label{fig:0-example} {\small (a) A Topcode-matrix $T_{code}(T)$; (b) a colored tree $T$; (c) a topological vector $V_{ec}(T)$ of the tree $T$.}}
\end{figure}

By the topological vector $V_{ec}(T_k)$ of each caterpillar $T_k$, the authors in \cite{Wang-Yao-Star-type-Lattices-2020, Yao-Wang-Su-Sun-ITOEC2020} show a \emph{topological coding lattice} as
\begin{equation}\label{eqa:topological-caterpillar-lattice}
\textrm{\textbf{L}}(\textbf{T}) =\left \{\sum^{m}_{k=1}a_k V_{ec}(T_k): a_k\in Z^0, k\in [1,m]\right \}
\end{equation} with $\sum^{m}_{k=1} a_k\geq 1$. The topological coding lattice $\textrm{\textbf{L}}(\textbf{T})$ is a part of traditional lattices \cite{Bernstein-Buchmann-dahmen-quantum-2009, Wang-Xiao-Yun-Liu-2014}.

We try to apply techniques of topological coding to graphic lattices, and make deep connections between  traditional lattices and graphic lattices for considering difficult problems in traditional lattices, in which many were proven to be NP-hard.

Ice-flower systems are basic and important in star-graphic lattices, we will show several ice-flower systems in building up star-graphic lattices, in researching graph structures, in particular total colorings.

\subsection{Operations, terminology}

The book \cite{Bondy-2008} and the survey article \cite{Gallian2019} show standard notation and terminology that will be used in this article, and all graphs used here are non-directed and have no multiple edges. Nevertheless, we will use some particular graphs and sets: A $(p,q)$-graph is a graph having $p$ vertices (small circles shown in Fig.\ref{fig:0-example} (b) in black or white) and $q$ edges (line-segments connecting two vertices), and its a vertex $x$ has degree $\textrm{deg}_G(x)=|N(x)|$, where $N(x)$ is the set of neighbors of the vertex $x$, and $|X|$ is the number of elements of a set $X$. Especially, a vertex $x$ is called a \emph{leaf} if its degree $\textrm{deg}_G(x)=1$. For the purpose of simplicity, we denote a consecutive set $\{a,a+1,a+2,\dots, b\}$ with two integers $a,b$ subject to $0<a<b$ as $[a,b]$, and $[a,b]^o=\{a,a+2,\dots, b\}$ with odd integers $a,b$ with respect to $1\leq a<b$. Let $n_d(G)$ be the number of vertices of degree $d$ in a graph $G$.

We say a graph $G$ to admit a \emph{$W$-type labelling} $f:S\subset V(G)\cup E(G)\rightarrow [a,b]$ if $f(x)\neq f(y)$ for any pair of distinct vertices $x,y\in V(G)$, and $G$ to admit a \emph{$W$-type coloring} $g:S\subset V(G)\cup E(G)\rightarrow [a,b]$ if $g(u)=g(v)$ for some two distinct vertices $u,v\in V(G)$, and rewrite color set $f(S)=\{f(w):w\in S\}$ hereafter.

We restate the operations that will be used  as follows:
\begin{defn}\label{defn:Leaf-splitting-coinciding-operations}
\cite{Wang-Yao-Star-type-Lattices-2020} \emph{Leaf-splitting and leaf-coinciding operations.} Let $uv$ be an edge of a graph $G$ with $\textrm{deg}_G(u)\geq 2$ and $\textrm{deg}_G(v)\geq 2$. A \emph{leaf-splitting operation} is defined as: Remove the edge $uv$ from $G$, the resulting graph is denoted as $G-uv$. Add a new leaf $v'$, and join it with the vertex $u$ of $G-uv$ by a new edge $uv'$, and then add another new leaf $u'$ to join it with the vertex $v$ of $G-uv$ by another new edge $vu'$, the resultant graph is written as $G\wedge uv$, see Fig.\ref{fig:2-leaf-splt} from (a) to (b).

Conversely, if a graph $H$ has two leaf-edges $uv'$ and $u'v$ with $\textrm{deg}_H(u)\geq 2$ and $\textrm{deg}_H(v')=1$, $\textrm{deg}_H(u')=1$ and $\textrm{deg}_H(v)\geq 2$, we coincide two edges $uv'$ and $u'v$ into one edge $uv=uv'\overline{\ominus} u'v$ with $u=u\odot u'$ and $v=v\odot v'$. The resultant graph is denoted as $H(uv'\overline{\ominus} vu')$, and the process of obtaining $H(uv'\overline{\ominus} vu')$ is called a \emph{leaf-coinciding operation}, see Fig.\ref{fig:2-leaf-splt} from (b) to (a).

If a graph $H_1$ has a leaf-edge $uv'$ with $\textrm{deg}_{H_1}(v')=1$ and another graph $H_2$ has a leaf-edge $u'v$ with $\textrm{deg}_{H_2}(u')=1$, we do a leaf-coinciding operation to two leaf-edges $uv'\in E(H_1)$ and $u'v\in E(H_2)$, and the resultant graph is denoted as $H_1\overline{\ominus} H_2$.\qqed
\end{defn}

\begin{figure}[h]
\centering
\includegraphics[width=12cm]{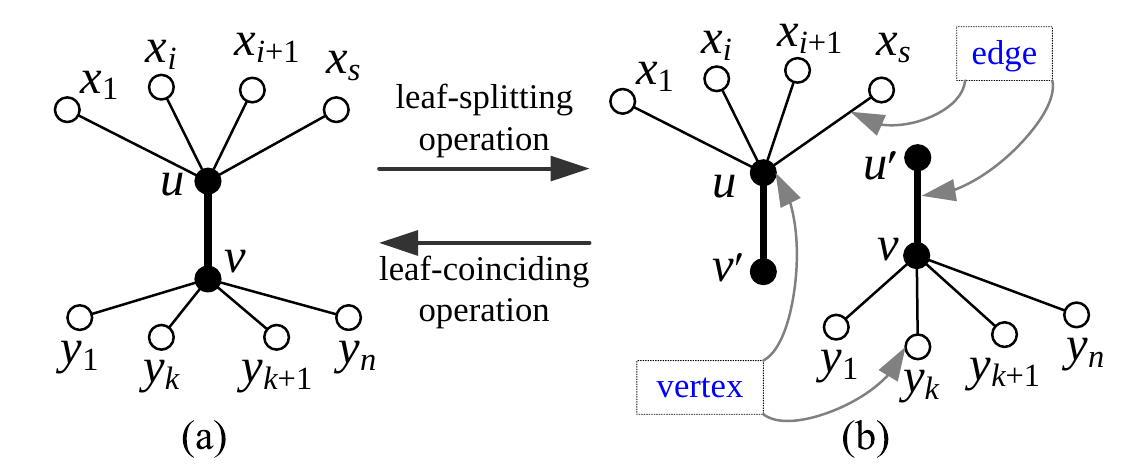}\\
\caption{\label{fig:2-leaf-splt} {\small A leaf-splitting operation is from (a) to (b); a leaf-coinciding operation is from (b) to (a).}}
\end{figure}

\begin{defn}\label{defn:colored-Leaf-splitting-coinciding-operations}
\emph{Colored leaf-splitting and leaf-coinciding operations.} Let $uv$ be an edge of a graph $G$ with $\textrm{deg}_G(u)\geq 2$ and $\textrm{deg}_G(v)\geq 2$, and let $G$ admit a proper total coloring $f$. We define a proper total coloring $g$ of the graph $G\wedge uv$ obtained by doing a leaf-splitting operation to the edge $uv$ of the graph $G$ as: $g(w)=f(w)$ for $w\in [V(G\wedge uv)\cup E(G\wedge uv)]\setminus \{u',v',uv',vu'\}$; $g(uv')=g(vu')=f(uv)$, $g(u)=f(u)$, $g(v)=f(v)$, $g(u')=f(u)$ and $g(v')=f(v)$. This process is called a \emph{colored leaf-splitting operation}, see a process from (a) to (b) in Fig.\ref{fig:2-leaf-splt-colored}.

The opposite operation of a colored leaf-splitting operation is called a \emph{colored leaf-coinciding operation} defined in the way: Let $u_iv_i$ with $i=1,2$ be two leaf-edges of a graph $L$ admitting a proper total coloring $\theta$, where $v_1$ and $v_2$ are leaves of $L$, and $\textrm{deg}_L(u_i)\geq 2$ with $i=1,2$. If $\theta(u_1)=\theta(v_2)$, $\theta(u_2)=\theta(v_1)$ and $\theta(u_1v_1)=\theta(u_2v_2)$, we can coincide two leaf-edges $u_1v_1$ and $u_2v_2$ into one edge $uv=u_1v_1\overline{\ominus} u_2v_2$ with $u=u_1\odot v_2$ and $v=u_2\odot v_1$, and the resultant graph is denoted as $L(u_1v_1\overline{\ominus} u_2v_2)$. Next, we define a proper total coloring $\psi$ for $L(u_1v_1\overline{\ominus} u_2v_2)$ as: $\psi(w)=\theta(w)$ for $w\in [V(L(u_1v_1\overline{\ominus} u_2v_2))\cup E(L(u_1v_1\overline{\ominus} u_2v_2))]\setminus \{u,v,uv\}$; $\psi(u)=\theta(u_1)=\theta(v_2)$, $\psi(v)=\theta(u_2)=\theta(v_1)$ and $\psi(uv)=\theta(u_1v_1)=\theta(u_2v_2)$, see Fig.\ref{fig:2-leaf-splt-colored} from (b) to (a). Particularly, if a colored graph $L_i$ has a colored leaf-edge $u_iv_i$ with $\textrm{deg}_{L_i}(v_i)=1$ with $i=1,2$, we do a colored leaf-coinciding operation to two colored leaf-edges $u_1v_1\in E(L_1)$ and $u_2v_2\in E(L_2)$, and the resultant graph is denoted as $L_1\overline{\ominus} L_2$.\qqed
\end{defn}

\begin{figure}[h]
\centering
\includegraphics[width=12cm]{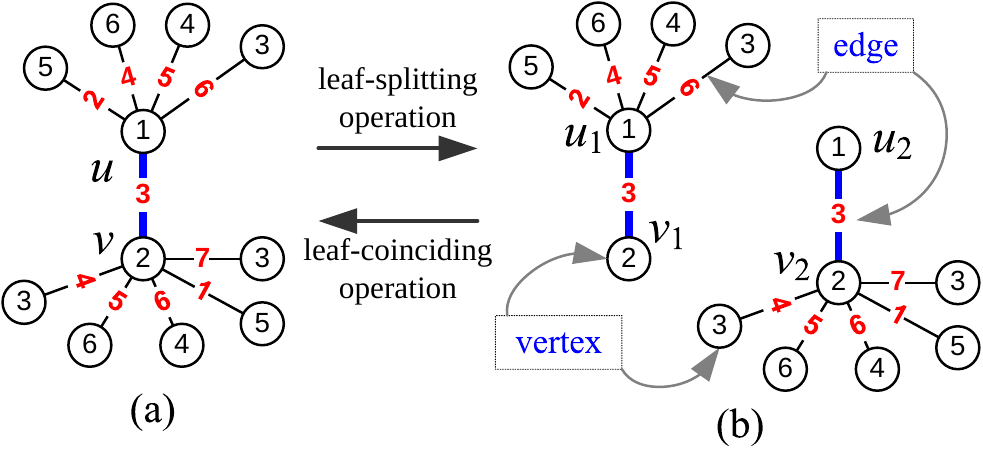}\\
\caption{\label{fig:2-leaf-splt-colored} {\small A colored leaf-splitting operation is from (a) to (b); a colored leaf-coinciding operation is from (b) to (a).}}
\end{figure}

\section{Ice-flower systems}
\textbf{Uncolored ice-flower system.} An \emph{ice-flower system} consists of stars $K_{1,m_1},K_{1,m_2},\dots, K_{1,m_n}$ and the leaf-coinciding operation. A \emph{star} $K_{1,m_i}$ with $m_i\geq 2$ is a complete bipartite connected graph with its vertex set $V(K_{1,m_i})=\{x_i,y_{i,j}:j\in [1,m_i]\}$ and its edge set $V(K_{1,m_i})=\{x_iy_{i,j}:j\in [1,m_i]\}$, so $K_{1,m_i}$ has $m_i$ leaves $y_{i,1},y_{i,2},\dots ,y_{i,m_i}$ and a unique non-leaf vertex $x_i$ with $\textrm{deg}_{K_{1,m_i}}(v)=m_i$. We call $n$ disjoint stars $K_{1,m_1},K_{1,m_2},\dots, K_{1,m_n}$ with $n\geq 1$ and $m_i\geq 2$ as an \emph{ice-flower system}, denoted as
\begin{equation}\label{eqa:ice-flower-base}
\textbf{\textrm{K}}=(K_{1,m_1},K_{1,m_2},\dots, K_{1,m_n})=(K_{1,m_j})^n_{j=1}
\end{equation}

The leaf-coinciding operation ``$K_{1,m_i}\overline{\ominus} K_{1,m_j}$'' on two stars $K_{1,m_i}$ and $K_{1,m_j}$ of an ice-flower system $\textbf{\textrm{K}}$ defined in (\ref{eqa:ice-flower-base}) is defined as doing a leaf-coinciding operation to an edge $x_iy_{i,a}$ of $K_{1,m_i}$ and an edge $x_jy_{j,b}$ of $K_{1,m_j}$, such that two edges $x_iy_{i,a}$ and $x_jy_{j,b}$ are coincided one edge $x_ix_j=x_iy_{i,a}\overline{\ominus} x_jy_{j,b}$, $x_i=x_i\odot y_{j,b}$ and $x_j=x_j\odot y_{i,a}$. Speaking simply, removing vertices $y_{i,a}$ and $y_{j,b}$ from $K_{1,m_i}$ and $K_{1,m_j}$ respectively, and adding a new edge, denoted as $x_ix_j$, joins $x_i$ with $x_j$ together. Thereby, an uncolored ice-flower system $\textbf{\textrm{K}}$ holds the leaf-coinciding operation ``$K_{1,m_i}\overline{\ominus} K_{1,m_j}$'' for any pair of stars $K_{1,m_i}$ and $K_{1,m_j}$, we call $\textbf{\textrm{K}}$ a \emph{strongly uncolored ice-flower system}.

\textbf{Colored ice-flower system.} If each star $K_{1,m_i}$ of an ice-flower system $\textbf{\textrm{K}}$ defined well in (\ref{eqa:ice-flower-base}) admits a $W_i$-type coloring $g_i$, we call $\textbf{\textrm{K}}$ a \emph{colored ice-flower system}, and rewrite it by \begin{equation}\label{eqa:colored-ice-flower-base}
\textbf{\textrm{K}}^c=(K^c_{1,m_1},K^c_{1,m_2},\dots, K^c_{1,m_n})=(K^c_{1,m_j})^n_{j=1}
\end{equation} for distinguishing. The colored leaf-coinciding operation ``$K_{1,m_i}\overline{\ominus} K_{1,m_j}$'' is defined in Definition \ref{defn:colored-Leaf-splitting-coinciding-operations}. If each pair of stars $K^c_{1,m_i}$ and $K^c_{1,m_j}$ holds the leaf-coinciding operation ``$K^c_{1,m_i}\overline{\ominus} K^c_{1,m_j}$'', we call $\textbf{\textrm{K}}^c$ a \emph{strongly colored ice-flower system} (see an example shown in Fig.\ref{fig:small-1}). Notice that a colored ice-flower system may be not strong, in general.

\section{Star-graphic lattices}

\subsection{Uncolored star-graphic lattices}By an ice-flower system $\textbf{\textrm{K}}$ defined in (\ref{eqa:ice-flower-base}), the following set
\begin{equation}\label{eqa:star-lattices}
\textbf{\textrm{L}}(\overline{\ominus} \textbf{\textrm{K}})=\left \{\overline{\ominus}|^n_{j=1}a_jK_{1,m_j}:~a_j\in Z^0\right \}
\end{equation}with $\sum ^n_{j=1}a_j\geq 1$ is called an \emph{uncolored star-graphic lattice}, so the ice-flower system $\textbf{\textrm{K}}$ defined in (\ref{eqa:ice-flower-base}) is called an \emph{uncolored star-graphic base}. It is easy to see that any uncolored connected graph $G$ can be expressed as $G=\overline{\ominus}|^n_{j=1}a_jK_{1,m_j}$ with $a_j\in Z^0$ and $\sum ^n_{j=1}a_j\geq 1$.

\begin{lem}\label{thm:ice-flower-system-any-graph}
Any uncolored connected graph $G$ can be expressed as $G=\overline{\ominus}|^n_{j=1}a_jK_{1,m_j}$ with $\sum ^n_{j=1}a_j\geq 1$ and $a_j\in Z^0$ by means of some ice-flower system $\textbf{\textrm{K}}$ defined in (\ref{eqa:ice-flower-base}) and the leaf-coinciding operation, where $\textrm{deg}_G(v_i)=d_i=m_i$ for $v_i\in V(G)$ and $\textrm{deg}_G(v_i)\geq 2$.
\end{lem}
\begin{proof} Let $G$ be a uncolored connected graph, and let $L(G)$ be the set of leaves of $G$. If the graph $V(G-L(G))$ obtained by removing $L(G)$ from $G$ holds $|V(G-L(G))|=1$, so $G$ is a star, we are done. If $|V(G-L(G))|\geq 2$, we take arbitrarily a vertex $w\in V(G-L(G))$, so $\textrm{deg}_G(w)\geq 2$. Let $N(w)=\{u_1,u_2,\dots ,u_d\}$, where $d=\textrm{deg}_G(w)$. We do a series of leaf-coinciding operations to $G$, such that the resultant graph $G'$ has a star $K_{1,d}$ and another component $H$, where $K_{1,d}$ has its vertex set $\{w\}\cup N(w)$ and edge set $\{wu_i:i\in [1,d]\}$, and $H$ has its leaf-set $L(H)=L(G)\cup \{w_1,w_2,\dots ,w_d\}$ and $\{u_1w_1,u_2w_2,\dots ,u_dw_d\}\subset E(H)$.

Removing the star $K_{1,d}$ from $G'$, we get a connected graph $H=G'-V(K_{1,d})$. Because of $|V(H-L(H))|=|V(G-L(G))|-1$, so we have $H=\overline{\ominus}|^{n-1}_{j=1}a_jK_{1,m_j}$ with $\sum ^{n-1}_{j=1}a_j=|V(H-L(H))|$ and $a_j\in Z^0$. Immediately, $G=\overline{\ominus}|^{n}_{j=1}a_jK_{1,m_j}$ with $\sum ^{n}_{j=1}a_j=|V(G-L(G))|$ and $a_j\in Z^0$, where $a_n=1$ and $K_{1,m_n}=K_{1,d}$, according to the hypothesis of mathematical induction.
\end{proof}

Each star $K_{1,m_i}$ of an ice-flower system $\textbf{\textrm{K}}$ has its vertex set $V(K_{1,m_i})=\{x_i,y_{i,j}:j\in [1,m_i]\}$ and its edge set $V(K_{1,m_i})=\{x_iy_{i,j}:j\in [1,m_i]\}$ with $m_i\geq 2$, so $K_{1,m_i}$ has its own leaf set $L(K_{1,m_i})=\{y_{i,1},y_{i,2},\dots ,y_{i,m_i}\}$ and its unique non-leaf vertex $x_i$ with $\textrm{deg}_{K_{1,m_i}}(x_i)=m_i$. Suppose that $(m_{1}$, $m_{2},\dots ,m_{n})$ obeys: A sequence $\textbf{\textrm{d}}=(m_1$, $ m_2, \dots , m_n)$ with $m_{i}\leq m_{i+1}$ to be the degree sequence of a certain graph $G$ of $n$ vertices if and only if $\sum^n_{i=1}m_i$ is even and
\begin{equation}\label{eqa:c3xxxxx}
\sum^k_{i=1}m_i\leq k(k-1)+\sum ^n_{i=k+1}\min\{k,m_i\}
\end{equation}
shown by Erd\"{o}s and Gallai in 1960 \cite{Bondy-2008}.

We leaf-coincide two leaf-edges $x_iy_{i,m_{i}}$ and $x_{i+1}y_{i+1,1}$ for each pair of $K_{1,m_{i}}$ and $K_{1,m_{i+1}}$ for $i\in [1,n-1]$ into one edge $x_ix_{i+1}$ joining $K_{1,m_{i}}$ and $K_{1,m_{i+1}}$ together, such that $x_i \odot y_{i+1,1}$ and $x_{i+1}\odot y_{i,m_{i}}$. Then we get a caterpillar $T$ for its ride $P=x_1x_2\dots x_n$ with $x_i\in V(K_{1,m_{i}})$ having vertex set $V(K_{1,m_i})=\{x_i,y_{i,j}:j\in [1,m_i]\}$, next we leaf-coincide two leaf-edges $x_1y_{1,1}$ and $x_{n}y_{n,m_{n}}$ into one edge $x_1x_{n}$ with $x_1\odot y_{n,m_{n}}$ and $x_{n}\odot y_{1,1}$, the resultant graph, denoted as $T^*$, is like a \emph{haired-cycle}.
We write $T^*=\overline{\ominus}^n_{j=1}K_{1,m_{j}}$, and then do some leaf-coinciding operations on some pairs of leaf-edges of $T^*$ to get a connected graph $G$, such that $G$ has no leaf. Thereby, $G$ has a cycle $C=x_1x_2\dots x_nx_1$ containing each vertex of $G$, that is, $G$ is \emph{hamiltonian}. Especially, we denote $G$ by
\begin{equation}\label{eqa:hamiltonian}
G=\overline{\ominus} T^*=\overline{\ominus}[\overline{\ominus}^n_{j=1}K_{1,m_{j}}]=\overline{\ominus}^2|^n_{j=1}K_{1,m_{j}}.
\end{equation} Notice that there are two or more hamiltonian graphs like $G$ by leaf-coinciding some pairs of leaf-edges of $T^*$. On the other hand, each permutation $k_{1}k_{2}\cdots k_{n}$ of $m_{1}m_{2}\cdots m_{n}$ distributes us a set of hamiltonian graphs $\overline{\ominus}^2|^n_{j=1}K_{1,k_{j}}$. As known, there are $M_p(=n!)$ permutations, we have $M_p$ ice-flower systems $\textbf{\textrm{K}}_k=(K_{1,k_1},K_{1,k_2},\dots, K_{1,k_n})$ with $k\in [1,M_p]$, and each ice-flower system $\textbf{\textrm{K}}_k$ induces a set of hamiltonian graphs $\overline{\ominus}^2|^n_{j=1}K_{1,k_{j}}$, we write this set as $\overline{\ominus}^2\textbf{\textrm{K}}_k$ and $P_{ermu}(\textbf{\textrm{K}})$ to be the set of $M_p$ ice-flower systems $\textbf{\textrm{K}}_k$. Then the following set
\begin{equation}\label{eqa:hamiltonian-star-lattices}
\textbf{\textrm{L}}(\overline{\ominus}^2 P_{\textrm{ermu}}(\textbf{\textrm{K}}))=\left \{\overline{\ominus}^2|^{M_p}_{k=1}a_k\textbf{\textrm{K}}_k,~\textbf{\textrm{K}}_k\in P_{\textrm{ermu}}(\textbf{\textrm{K}})\right \}
\end{equation} with $\sum^{M_p}_{k=1}a_k=1$ a \emph{hamiltonian star-graphic lattice}. Clearly, a connected graph $H$ with $n$ vertices and degree sequence $(m_{1},m_{2},\dots ,m_{n})$ with $2\leq m_{i}\leq m_{i+1}$ is hamiltonian if and only if $H\in \textbf{\textrm{L}}(\overline{\ominus}^2 P_{\textrm{ermu}}(\textbf{\textrm{K}}))$. The process of showing this fact is just an inverse process of building up the hamiltonian star-graphic lattice $\textbf{\textrm{L}}(\overline{\ominus}^2 P_{\textrm{ermu}}(\textbf{\textrm{K}}))$, since each hamiltonian graph can be leaf-split into a haired-cycle $T^*$ obtained above according to Lemma \ref{thm:ice-flower-system-any-graph}.

We can construct graphs $G=\overline{\ominus}^{m+1}|^{M_p}_{k=1}a_k\textbf{\textrm{K}}_k$ with $\textbf{\textrm{K}}_k\in P_{\textrm{ermu}}(\textbf{\textrm{K}})$ containing $m$ edge-disjoint Hamilton cycles, also, we can express an Euler graph $H$ with degree sequence $\textbf{\textrm{d}}=(2m_1$, $ 2m_2, \dots , 2m_n)$ as $H=\overline{\ominus} |^{n}_{j=1}K_{1,2m_j}$, and so on.

\subsection{Colored star-graphic lattices} An ice-flower system $\textbf{\textrm{K}}^c$ admits a \emph{leaf-joining $W$-type coloring} $f$ if a connected graph $\overline{\ominus}|^n_{j=1}a_jK^c_{1,m_j}$ admits this $W$-type coloring $f$, where each star $K^c_{1,m_i}$ of $\textbf{\textrm{K}}^c$ admits this $W$-type coloring too. By Definition \ref{defn:colored-Leaf-splitting-coinciding-operations}, doing a colored leaf-coinciding operation to two disjoint $K^c_{1,m_i}$ and $K^c_{1,m_j}$ produces a connected graph, denoted as $K^c_{1,m_i}\overline{\ominus} K^c_{1,m_j}$. We have a \emph{$W$-type coloring star-graphic lattice}
\begin{equation}\label{eqa:leaf-joining-star-lattices}
\textbf{\textrm{L}}(\overline{\ominus} (f)\textbf{\textrm{K}}^c)=\left \{\overline{\ominus}(f)|^n_{j=1}a_jK^c_{1,m_j}:~a_j\in Z^0\right \}
\end{equation}
with $\sum ^n_{j=1}a_j\geq 1$ and the \emph{lattice base} $\textbf{\textrm{K}}^c$ admitting a leaf-joining $W$-type coloring $f$.

We are interesting on that each graph of the $W$-type coloring star-graphic lattice $\textbf{\textrm{L}}(\overline{\ominus} (f)\textbf{\textrm{K}}^c)$ admits this $W$-type coloring $f$.

\textbf{Planar star-graphic lattices.} An ice-flower system $\textbf{\textrm{K}}^{4c}$ defined in (\ref{eqa:colored-ice-flower-base}) admits a leaf-joining vertex coloring $h$ defined as: Each $K^{4c}_{1,m_j}$ of $\textbf{\textrm{K}}^{4c}$ admits a proper vertex coloring $h_j$ such that $h_j(x)\neq h_j(y)$ for any edge $xy$ of $K^{4c}_{1,m_j}$, and $h(V(K^{4c}_{1,m_j}))=[1,4]$ for $j\in [1,n]$. The leaf-coinciding operation ``$K^{4c}_{1,m_i}\overline{\ominus} K^{4c}_{1,m_j}$'' is defined in Definition \ref{defn:colored-Leaf-splitting-coinciding-operations}. Suppose that each graph $G=\overline{\ominus}_{\textrm{pla}}|^n_{j=1}a_jK^{4c}_{1,m_j}$ is connected, planar and $n_1(G)=0$. We obtain a \emph{planar star-graphic lattice}
\begin{equation}\label{eqa:c3xxxxx}
\textbf{\textrm{L}}(\overline{\ominus}_{\textrm{pla}} \textbf{\textrm{K}}^{4c})=\{\overline{\ominus}_{\textrm{pla}}|^n_{j=1}a_jK^{4c}_{1,m_j},~K^{4c}_{1,m_j}\in \textbf{\textrm{K}}^{4c}\}
\end{equation} with $\sum^n_{j=1}a_j\geq 1$, and $\textbf{\textrm{K}}^{4c}$ is called a \emph{4-color lattice base}. In general, a graph $\overline{\ominus}|^n_{j=1}a_jK^{4c}_{1,m_j}$ is 4-colorable, but may be not planar.

As example, the star-decomposition (star-coincidence) of a 4-colorable planar graph $G=\overline{\ominus} (8K^{4c}_{1,5}, 2K^{4c}_{1,4})$ $\in \textbf{\textrm{L}}(\overline{\ominus}_{\textrm{pla}} \textbf{\textrm{K}}^{4c})$ is shown in Fig.\ref{fig:4-color-1} and Fig.\ref{fig:4-color-3}, and this 4-colorable planar graph $G=\overline{\ominus}^2(8K^{4c}_{1,5}, 2K^{4c}_{1,4})$ is hamiltonian.

\begin{figure}[h]
\centering
\includegraphics[width=16.2cm]{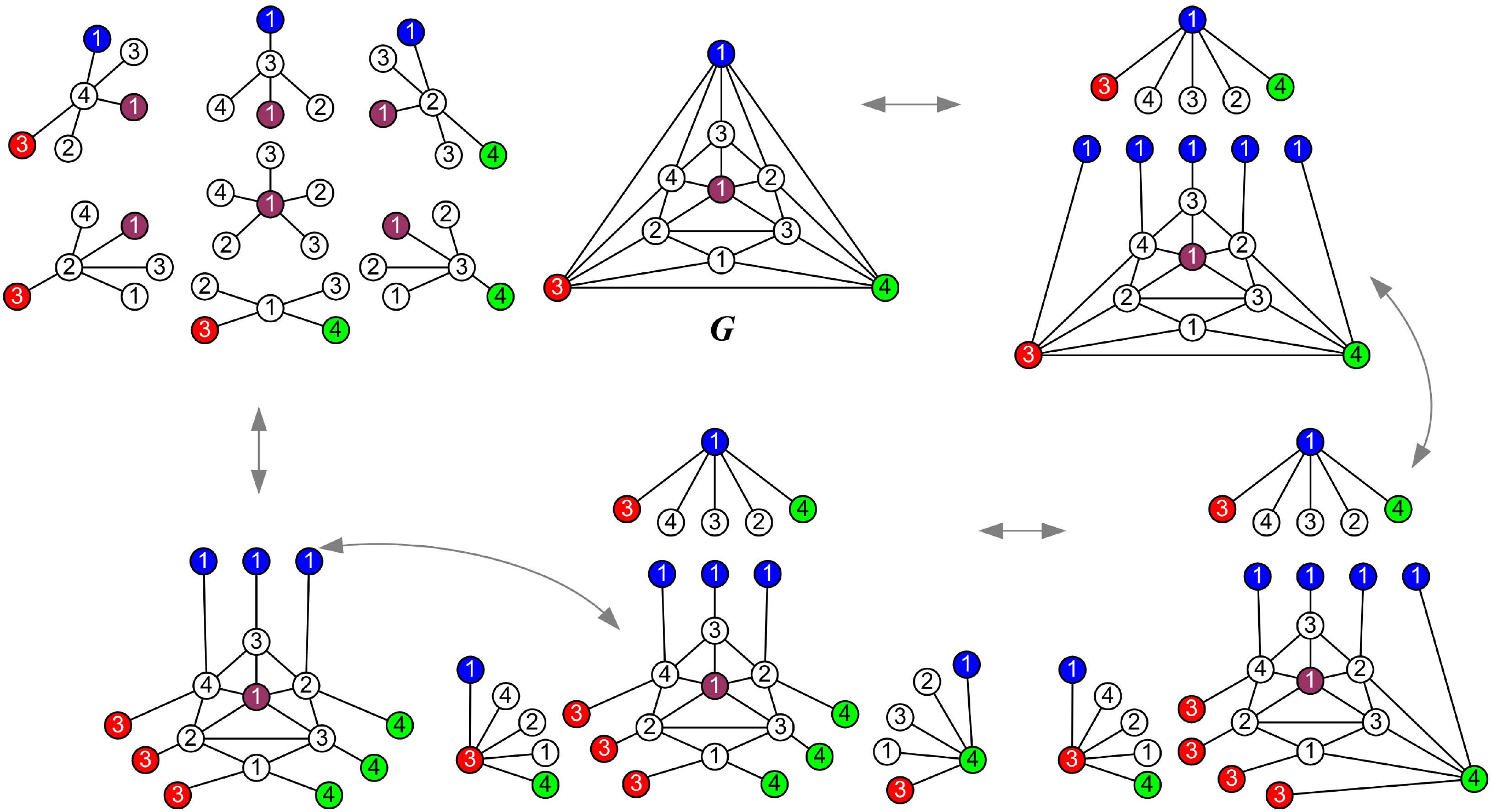}\\
\caption{\label{fig:4-color-1} {\small A scheme of decomposition of a planar graph $G$.}}
\end{figure}

\begin{figure}[h]
\centering
\includegraphics[width=16.2cm]{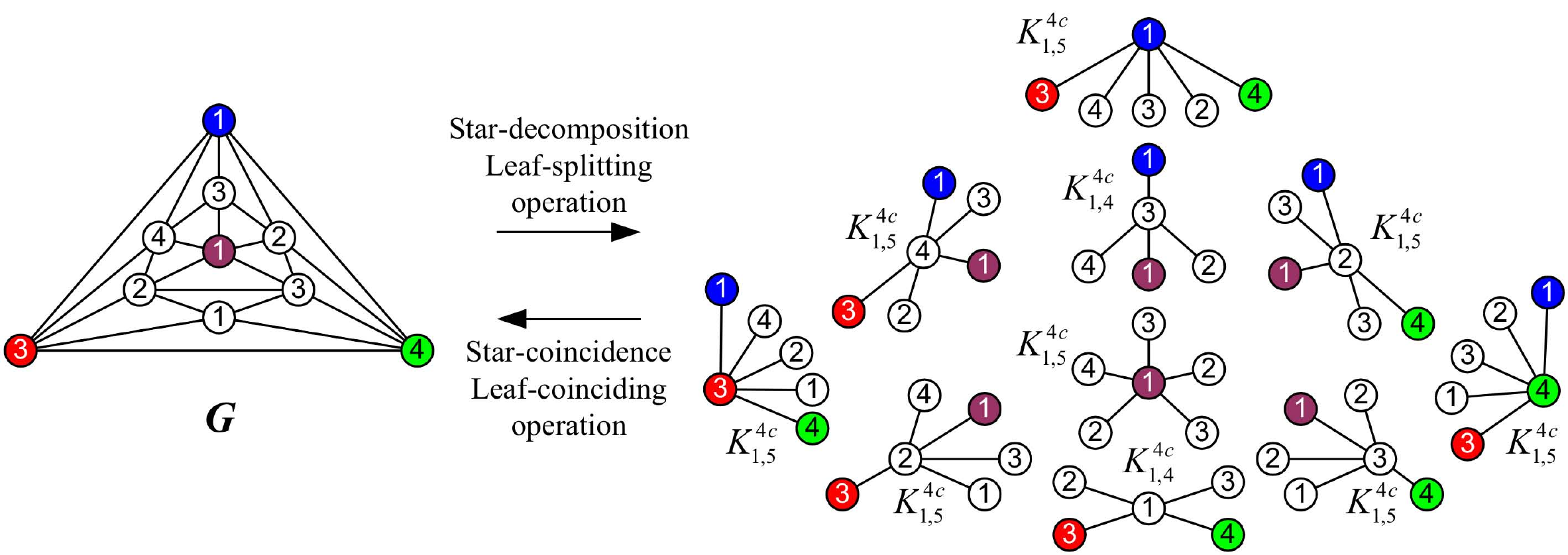}\\
\caption{\label{fig:4-color-3} {\small $G=\overline{\ominus} (8K^{4c}_{1,5}, 2K^{4c}_{1,4})$.}}
\end{figure}

\textbf{Spanning star-graphic lattices.} A connected graph $T$ is a tree if and only if $n_1(T)=2+\Sigma_{d\geq 3}(d-2)n_d(T)$ (\cite{Yao-Zhang-Yao-2007, Yao-Zhang-Wang-Sinica-2010}).

We define a leaf-joining vertex coloring $g$ for an ice-flower system $\textbf{\textrm{K}}^c$ defined in (\ref{eqa:colored-ice-flower-base}) as: Each $K^c_{1,m_j}$ of $\textbf{\textrm{K}}^c$ admits a proper vertex coloring $g_j$ such that $g_j(x)\neq g_j(y)$ for any pair of vertices $x,y$ of $K^c_{1,m_j}$, and each leaf-coinciding graph $T=\overline{\ominus}|^n_{j=1}K^c_{1,m_j}$ is connected based on the leaf-coinciding operation ``$K^c_{1,m_i}\overline{\ominus} K^c_{1,m_j}$'' defined in Definition \ref{defn:colored-Leaf-splitting-coinciding-operations}, such that

(1) $n_1(T)=2+\Sigma_{d\geq 3}(d-2)n_d(T)$ holds true;

(2) $T$ admits a proper vertex coloring $f=\overline{\ominus}|^n_{j=1}g_j$ with $f(u)\neq f(w)$ for any pair of vertices $u,w$ of $T$.

We get a set $\textbf{\textrm{L}}(\overline{\ominus} (m,g)\textbf{\textrm{K}}^c)$ containing the above leaf-coinciding graphs $T=\overline{\ominus}|^n_{j=1}K^c_{1,m_j}$ if $|V(T)|=m$. Since each graph $T\in\textbf{\textrm{L}}(\overline{\ominus} (m,g)\textbf{\textrm{K}}^c)$ is a tree, and Cayley's formula $\tau(K_m)=m^{m-2}$ in graph theory (Ref. \cite{Bondy-2008}) tells us the number of elements of $\textbf{\textrm{L}}(\overline{\ominus} (n)\textbf{\textrm{K}}^c)$ to be equal to $m^{m-2}$. We call this set $\textbf{\textrm{L}}(\overline{\ominus} (m,g)\textbf{\textrm{K}}^c)=\{\overline{\ominus}|^n_{j=1}K^c_{1,m_j},~K^c_{1,m_j}\in \textbf{\textrm{K}}^c\}$  a \emph{spanning star-graphic lattice}.

\vskip 0.2cm

\textbf{Uniform ice-flower systems.}  The authors in \cite{Wang-Su-Yao-2019, Wang-Yao-Star-type-Lattices-2020, Yao-Wang-Su-Ma-Wang-Sun-ITNEC-2020} have introduced some articular proper total colorings, in which, for a graph $G$, a proper total coloring $f:V(G)\cup E(G)\rightarrow [1,M]$ induces a parameter
{\small
\begin{equation}\label{eqa:felicitous-difference-total-coloring}
B_{fdt}(G,f,M)=\max_{uv \in E(G)}\{c_f(uv)\}-\min_{uv \in E(G)}\{c_f(uv)\}
\end{equation}
}where $c_f(uv)=|f(u)+f(v)-f(uv)|$. If $B_{fdt}(G,f,M)=0$, we call $f$ a \emph{felicitous-difference proper total coloring} of $G$, the number $\chi''_{fdt}(G)=\min_f \{M:~B_{fdt}(G,f,M)=0\}$ over all felicitous-difference proper total colorings of $G$ is called \emph{felicitous-difference total chromatic number}. See some felicitous-difference proper total colorings shown in Fig.\ref{fig:felicitous-di-constants}.

\begin{figure}[h]
\centering
\includegraphics[width=16cm]{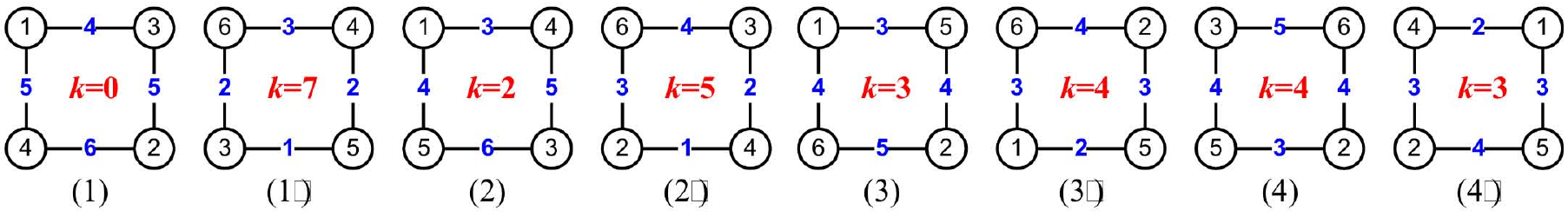}\\
\caption{\label{fig:felicitous-di-constants} {\small Various constants $k=|f(u)+f(v)-f(uv)|$ of felicitous-difference proper total colorings.}}
\end{figure}

There are four uniform ice-flower systems $(F_kD_{1,6})^6_{k=1}$, $(F_sD_{1,6})^{12}_{s=7}$, $(F_kK_{1,6})^{6}_{k=1}$ and $(F_sF_{1,6})^{12}_{s=7}$ shown in Fig.\ref{fig:felicitous-di-1}, where $(F_kD_{1,6})^6_{k=1}$ and $(F_kK_{1,6})^{6}_{k=1}$ are mutually \emph{dual}, and $(F_sD_{1,6})^{12}_{s=7}$ and $(F_sF_{1,6})^{12}_{s=7}$ are \emph{dual} to each other.

\begin{figure}[h]
\centering
\includegraphics[width=14cm]{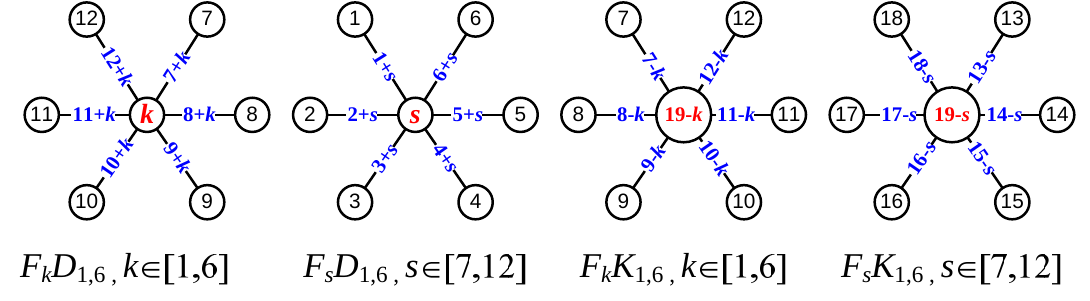}\\
\caption{\label{fig:felicitous-di-1} {\small Four uniform ice-flower systems.}}
\end{figure}

If a graph $H$ has its vertex degree $\textrm{deg}_H(x)=1$ or $\textrm{deg}_H(x)=\Delta(H)$ for each vertex $x\in V(H)$, then we call $H$ a $\Delta$-saturated graph. See two $\Delta$-saturated graphs shown in Fig.\ref{fig:small-2}, where (a) admits a \emph{graph homomorphism} to (b). By a uniform  ice-flower system $(F_kD_{1,m})^m_{k=1}$ (each $F_kD_{1,m}$ is isomorphic to a star $K_{1,m}$, see an example shown in Fig.\ref{fig:felicitous-di-1}) and another uniform ice-flower system $(SF_{1,m}D_k)^m_{k=1}$ (each $SF_{1,m}D_k$ is isomorphic to a star $K_{1,m}$, see an example shown in Fig.\ref{fig:small-1}), we can show
\begin{thm} \label{thm:felicitous-difference-results}
(1) $\chi''_{fdt}(K_{m,m})=3m$.

(2) Each bipartite graph $G$ holds $\chi''_{fdt}(G)\leq 3\Delta(G)$.

(3) Each tree $T$ holds $\chi''_{fdt}(T)\leq 1+2\Delta(T)$.
\end{thm}

\begin{figure}[h]
\centering
\includegraphics[width=16cm]{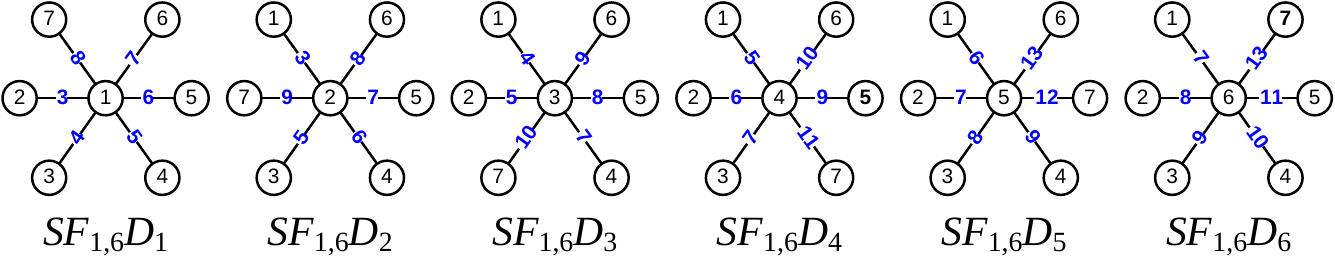}\\
\caption{\label{fig:small-1} {\small A strongly colored uniform ice-flower system $(SF_{1,6}D_k)^6_{k=1}$.}}
\end{figure}

\begin{figure}[h]
\centering
\includegraphics[width=14cm]{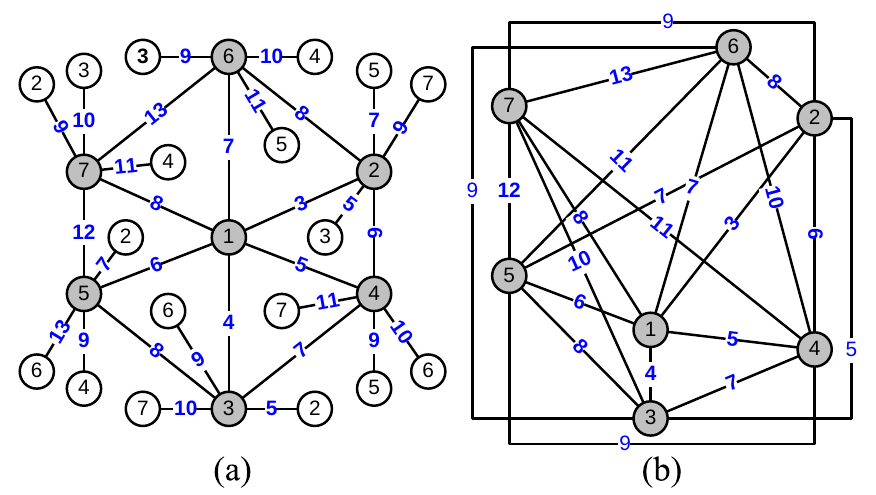}\\
\caption{\label{fig:small-2} {\small Two $\Delta$-saturated graphs obtained from the uniform ice-flower system $(SF_{1,6}D_k)^6_{k=1}$ shown in Fig.\ref{fig:small-1}.}}
\end{figure}

\section{Conclusion}

We have defined the uniform ice-flower systems $\textbf{\textrm{K}}$, $\textbf{\textrm{K}}^c$, $(F_kD_{1,m})^m_{k=1}$, $(SF_{1,m}D_k)^m_{k=1}$, \emph{etc.}, and defined uncolored or colored leaf-splitting and leaf-coinciding operations. These ice-flower systems and the leaf-coinciding operation help us to construct the spanning star-graphic lattice $\textbf{\textrm{L}}(\overline{\ominus} (m,g)\textbf{\textrm{K}}^c)$, the planar star-graphic lattice $\textbf{\textrm{L}}(\overline{\ominus}_{\textrm{pla}} \textbf{\textrm{K}}^{4c})$ and hamiltonian star-graphic lattice $\textbf{\textrm{L}}(\overline{\ominus}^2 P_{\textrm{ermu}}(\textbf{\textrm{K}}))$. More importantly, we have computed the number $m^{m-2}$ of elements of a star-graphic lattice $\textbf{\textrm{L}}(\overline{\ominus} (n)\textbf{\textrm{K}}^c)$.

For further investigation, we have the following problems: (1)  \emph{Determine} uniformly colored uniform ice-flower systems with felicitous-difference proper total colorings for closing $\chi''_{fdt}(G)$. (2) \emph{Find} solutions for Decomposition Number String Problem. (3) \emph{Build} more connections between traditional lattices and star-graphic lattices. (4) \emph{Compute} the number of elements of a star-graphic lattice.

\section*{Acknowledgment}

The author, \emph{Bing Yao}, was supported by the National Natural Science Foundation of China under grant No. 61163054, No. 61363060 and No. 61662066. The author, \emph{Hongyu Wang}, thanks gratefully the National Natural Science Foundation of China under grants No. 61902005, and China Postdoctoral Science Foundation Grants No. 2019T120020 and No. 2018M641087.


\end{document}